\documentclass[11pt]{article}
\usepackage{amsmath,latexsym,amsthm,amsfonts}
\usepackage{amssymb}

\newtheorem{theorem}{Theorem}[section]
\newtheorem{lemma}[theorem]{Lemma}

\begin{document}

\newcommand{\GF}{\mathrm{GF}}
\newcommand{\F}{\mathbf{F}}
\newcommand{\bO}{\mathrm{O}}

\newcommand{\ex}{\mathrm{ex}}
\newcommand{\la}{\langle}
\newcommand{\ra}{\rangle}
\newcommand{\pr}{\prime}

\title{A construction for bipartite Tur\'an numbers}
\author{Ivan Livinsky}
\date{}
\maketitle

\begin{abstract}
We consider in detail the well-known family of graphs $G(q,t)$ that establish an asymptotic lower bound for Tur\'an numbers $\ex(n,K_{2,t+1})$. We prove that $G(q,t)$ for some specific $q$ and $t$ also gives an asymptotic bound for $K_{3,3}$ and for some higher complete bipartite graphs as well. The asymptotic bounds we prove are the same as provided by the well-known Norm-graphs.
\end{abstract}

\section{Introduction}
In his 1996 paper \cite{F1} for a prime power $q$ and $t\mid q-1$ F\"uredi introduced graphs $G(q,t)$ 
defined as follows. Let
$\F=\GF(q)$ be a finite field having $q$ elements. Let $H\subseteq \F^\ast$ be a subgroup of the multiplicative group
containing $t$ elements. Define 
$$
V(G(q,t))=(\F\times \F \setminus\{(0, 0)\}) / \sim
$$
 where two pairs $(a_1, b_1)$, $(a_2, b_2)$ are equivalent iff there exists $h\in H$ such that $a_1=ha_2$ and $b_1=hb_2$. Hence, $G(q,t)$ has $(q^2-1)/t$ vertices.
We write $\la a,b\ra$ for the equivalence class containing pair $(a,b)$. Two vertices $\la a,b\ra$ and $\la x,y\ra$ are joined by an edge iff
$$
ax+by\in H.
$$
We can easily see that in this way a simple graph is correctly defined. For a pair $(a,b)$ and $h\in H$ fixed the equation
$$
ax+by=h
$$
defines a line containing $q$ points in $\F^2$ and any two points on this line are pairwise non-equivalent. Therefore, degree of each vertex is either $q$ or $q-1$. Thus, $G(q,t)$ contains at least $\frac{1}{2t}(q^2-1)(q-1)$ edges.

\pagebreak
\begin{theorem}[\cite{F1}]\label{K2_t+1}
For arbitrary prime power $q$  and $t\mid q-1$ the graph $G(q,t)$ is $K_{2,t+1}$-free.
\end{theorem}
We reproduce F\"uredi's proof here since we will use some of its steps later.
\begin{proof}
If a vertex $\la x,y\ra$ is attached to two distinct vertices $\la a_1,b_1\ra$ and $\la a_2,b_2 \ra $ then there exist $h_1,h_2\in H$ such that the following system of linear equations
\begin{align*}
a_1x+b_1y &= h_1,\\
a_2x+b_2y &= h_2
\end{align*}
holds.

First, we show that the matrix of the system
$$
\left(
\begin{array}{cc}
a_1 & b_1\\
a_2 & b_2\\
\end{array}
\right)
$$
has full rank. Indeed, both lines are nonzero and if $a_1=ca_2$ and $b_1=cb_2$ for some $c\in\F^\ast$ then we have also $h_1=ch_2$ implying that $c=h_1h_2^{-1}\in H$. Thus, we have that $\la a_1,b_1\ra = \la a_2,b_2\ra$ and we arrive to a contradiction.

Therefore, for arbitrary $h_1$ and $h_2$ there exists a unique solution $(x,y)$. We have $t^2$ choices for $h_1$ and $h_2$ totally, and all the solutions are divided into $t$ classes. Therefore, there are at most $t$ vertices attached to both $\la a_1,b_1\ra$ and $\la a_2,b_2\ra$.
\end{proof}

We show that for some specific choice of $q$ and $t$ the graphs $G(q, t)$ in a similar way give lower bounds for larger bipartite graphs as well. In fact, using them we can obtain the same asymptotic bounds as given by the well-known Norm-graphs \cite{N2}.

\section{Construction for $K_{3,3}$}
Let $q$ be a prime power. Consider the graph $G=G(q^2,q+1)$. Let $\F=\GF(q^2)$.

\begin{theorem}\label{K33}
For arbitrary prime power $q$ the graph $G(q^2,q+1)$ is $K_{3,3}$-free.
\end{theorem}
We will need an auxiliary lemma.

\begin{lemma}\label{L}
Let $H\subset\F^\ast$, $|H|=q+1$, be a subgroup. Let $a,b\in\F$, $a\neq 0$, $b\neq 0$. Then the equation
$$
ax+by=1
$$
has at most two solutions $(x,y)$ for $x,y\in H$.
\end{lemma}
\begin{proof}
Assume that $(x,y)$ is a solution. Then $by = 1-ax$, $x^{q+1}=y^{q+1}=1$ and
\begin{align*}
b^{q+1}&=(by)^{q+1} = (1-ax)^{q+1} = (1-ax)^q(1-ax)\\
 &= (1-a^qx^q)(1-ax) = (1-\tfrac{a^q}{x})(1-ax) 
= 1 - ax -\frac{a^q}{x} +a^{q+1}.
\end{align*}
Therefore, $x$ is a solution to a proper quadratic equation
$$
ax^2 - (a^{q+1}- b^{q+1} + 1) x +  a^q =0,
$$
and $y= \frac{1-ax}{b}$.

Therefore, there are at most two possible pairs $(x,y)$.
\end{proof}

\begin{proof}[Proof of Theorem~\ref{K33}]
Consider three distinct vertices $\la a_1,b_1\ra$, $\la a_2,b_2\ra$, $\la a_3,b_3\ra$ and assume that another vertex $\la x,y\ra$ is attached to all of them. We have a system of equations
\begin{align*}
a_1x+b_1y = h_1,\\
a_2x+b_2y = h_2,\\
a_3x+b_3y = h_3.
\end{align*}
We know from the proof of Theorem~\ref{K2_t+1} that the matrix 
$$
\left(
\begin{array}{cc}
a_1 & b_1\\
a_2 & b_2\\
\end{array}
\right)
$$
has rank two. Therefore, the third equation must be a linear combination of the first two. That is, there must exist uniquely defined coefficients $\alpha,\beta\in\F$ such that
\begin{align*}
\alpha a_1+\beta a_2 &= a_3,\\
\alpha b_1+\beta b_2 &= b_3,\\
\alpha h_1+\beta h_2 &= h_3.
\end{align*}
Moreover, $\alpha\neq 0$ and $\beta\neq 0$ since otherwise we will have equality between the initial vertices. Consider the last equation. Let $r=h_1h_3^{-1}$ and $s=h_2h_3^{-1}$. Then
$$
\alpha r+\beta s =1.
$$
However, according to Lemma~\ref{L} there are at most two solutions $(r,s)$ to this equation. Therefore, there are at most $2(q+1)$ triples $(h_1,h_2,h_3)$ such that the original system has a solution. These triples define at most two vertices $\la x,y\ra$. Therefore, $G$ is $K_{3,3}$-free.
\end{proof}
We have that $|V(G)| = \frac{q^4-1}{q+1}=(q^2+1)(q-1)=q^3-q^2+q-1$. We can also give an exact formula for the number of edges.
\begin{theorem}
Let $G=G(q^2,q+1)$. If $q$ is odd then $|E(G)|=\frac{1}{2}(q^5 - q^4 + q^3 - 2q^2 + 1)$. 
If $q=2^k$ then 
$|E(G)|=\frac{1}{2}(q^5 - q^4 + q^3 - 2q^2)$.
\end{theorem}
\begin{proof}
If $q$ is odd then $q^2\equiv 1\ (\mathrm{mod}\ 4)$ and $-1$ is a square in $\F$. Therefore, the equation
$$
x^2+y^2=c
$$
has exactly $|\F|-1=q^2-1$ solutions for all $c\neq 0$. If $q=2^k$ then this equation defines a line in $\F^2$ and has $|\F|=q^2$ solutions.

A vertex $\la x,y\ra$ has degree $q^2-1$ in $G$ iff $x^2+y^2\in H$ according to the construction.
 Therefore, the number of vertices of degree $q^2-1$ is equal to $q^2-1$ for $q$ odd, and to $q^2$ for $q$ even. Thus, for $q$ odd
\begin{align*}
|E(G)|&=\frac{1}{2}\left( (q^3-2q^2+q) q^2 + (q^2-1)^2\right) \\
&= \frac{1}{2}(q^5 - q^4 + q^3 - 2q^2 + 1).
\end{align*}
And for $q$ even
\begin{align*}
|E(G)|&=\frac{1}{2}\left( (q^3-2q^2+q - 1) q^2 + q^2(q^2-1)\right) \\
&= \frac{1}{2}(q^5 - q^4 + q^3 - 2q^2).
\end{align*}
\end{proof}

Therefore, for $n=q^3-q^2+q-1$ we have a $K_{3,3}$-free graph with $\frac{1}{2}n^{\frac{5}{3}} + \frac{1}{3}n^{\frac{4}{3}}+\bO(n)$ edges. Together with the upper bound \cite{F2} for $K_{3,3}$ this gives the asymptotic formula
$$
\ex(n,K_{3,3})=\frac{1}{2}n^{\frac{5}{3}}(1+o(1)).
$$
We can also obtain a lower bound for the graphs $K_{3,2t^2+1}$.
\begin{theorem}\label{K3,2t^2+1}
For arbitrary prime power $q$ and $t\mid q-1$ the graph $G(q^2, t(q+1))$ is $K_{3,2t^2+1}$-free.
\end{theorem}
\begin{proof}
We follow the same steps as in the proof of Theorem~\ref{K33}. We arrive to the equation
$$
\alpha r+\beta s =1,
$$
where now we assume that $r,s$ belong to the subgroup $H$ of order $t(q+1)$. Let $H^\pr$ be a subgroup of order $q+1$. Then we can choose $t$ coset representatives $g_1,\ldots,g_t$ of $H / H^\pr$. We have that $r=g_i r^\pr$, $s=g_j s^\pr$ for some $r^\pr,s^\pr\in H^\pr$ and
$$
\alpha^\pr r^\pr+\beta^\pr s^\pr =1,
$$
where $\alpha^\pr = g_i\alpha$, $\beta^\pr = g_j\beta$. According to Lemma~\ref{L} this equation has at most two solutions. Since the choice of $g_i, g_j$ can be arbitrary we have that there are at most $2t^2$ pairs $(r,s)$. Therefore, $G$ is $K_{3,2t^2+1}$-free.
\end{proof}

This gives an asymptotic bound of the form
$$
\ex(n, K_{3,2t^2+1})\geq \frac{1}{2}t^{\frac{2}{3}} n^{\frac{5}{3}}(1+o(1)).
$$
This asymptotic bound was also proved by Mont\'agh in his PhD thesis \cite{M} using a factorization of the Brown graph \cite{B}.

\section{General case}
Alon, Koll\'ar, R\'onyai, and Szab\'o introduced Norm-graphs in the papers \cite{N1,N2}. They constructed a family of graphs that were $K_{r,(r-1)!+1}$-free and had $n$ vertices and $\frac{1}{2}n^{2-\frac{1}{r}}(1+o(1))$ edges. Their construction depended heavily on the following algebro-geometric lemma

\begin{lemma}[\cite{N1}]\label{AG}
Let $q$ be a prime power, let $\F=\GF(q^r)$, and let 
$$
N:\F\rightarrow\GF(q),\quad x\mapsto x^{1+q+\cdots + q^{r-1}}
$$
 be the norm map of $\F$ over $\GF(q)$. Let $c_1, \ldots, c_r, d_1,\ldots, d_r\in \F$ be some elements. If $d_i\neq d_j$ for $i\neq j$ then the system of equations
\begin{align*}
N(x+d_1)&=c_1,\\
N(x+d_2)&=c_2,\\
\ldots&\\
N(x+d_r)&=c_r,
\end{align*}
has at most $r!$ solutions for $x\in\F$.
\end{lemma}

Using Lemma~\ref{AG} we establish the same asymptotic bound using graphs $G(q,t)$.

\begin{theorem}\label{Kr(r-1)!+1}
For arbitrary prime power $q$ the graph $G(q^{r-1},q^{r-2}+\cdots + q + 1)$ is $K_{r,(r-1)!+1}$-free.
\end{theorem}
\begin{proof}
Let $\F=\GF(q^{r-1})$. Let $H\subset\F^\ast$ be a subgroup of order $q^{r-2} + \cdots +q + 1$.
Consider distinct $r$ vertices $\la a_1,b_1\ra$, $\ldots$, $\la a_r,b_r\ra$ and assume that another vertex $\la x,y\ra$ is attached to all of them. We have another system of linear equations
\begin{align*}
a_1x+b_1y &= h_1,\\
a_2x+b_2y &= h_2,\\
\ldots&\\
a_rx+b_ry &= h_r.
\end{align*}
As before, we have that all equations from third to last are linear combinations of the first two. That is, for every $j=3,\ldots,r$ there exist uniquely defined nonzero elements $\alpha_j, \beta_j\in\F$ such that
\begin{align*}
\alpha_j a_1+\beta_j a_2 &= a_j,\\
\alpha_j b_1+\beta_j b_2 &= b_j,\\
\alpha_j h_1+\beta_j h_2 &= h_j.
\end{align*}
Therefore, we have a system of $r-2$ equations
\begin{align*}
\alpha_3 h_1+\beta_3 h_2 &= h_3,\\
\ldots&\\
\alpha_r h_1+\beta_r h_2 &= h_r.
\end{align*}
Note that $N(h)=1$ for all $h\in H$. Therefore for each $j=3,\ldots,r$ we have that
$$
N\left(\frac{\alpha_j}{\beta_j}+ \frac{h_2}{h_1} \right) =N\left(\frac{h_j}{ h_1\beta_j}\right)
  = N(\beta_j)^{-1}.
$$
Moreover, we also have that
$$
N\left(\frac{h_2}{h_1}\right) = 1.
$$
Note that $\frac{\alpha_j}{\beta_j}\neq 0$, and  $\frac{\alpha_i}{\beta_i}\neq \frac{\alpha_j}{\beta_j}$ when $i\neq j$ since otherwise we would get $\la a_i,b_i\ra=\la a_j,b_j\ra$. 

We have a system of $r-1$ equations that satisfies the conditions of Lemma~\ref{AG}. Thus, it has at most $(r-1)!$ solutions for $h_2/h_1$. Each solution defines a unique vertex $\la x,y\ra$ attached to all $\la a_i,b_i\ra$. Therefore, $G$ is $K_{r, (r-1)!+1}$-free.
\end{proof}
We have that $G=G(q^{r-1},q^{r-2}+\cdots + q + 1)$ has $(q^{r-1}+1)(q-1)$ vertices and at least $\frac{1}{2}(q^{2r-2}-1)(q-1)$ edges. Therefore, it achieves the asymptotic lower bound of the form
$$
\ex(n, K_{r, (r-1)!+1})\geq \frac{1}{2}n^{2-\frac{1}{r}}(1+o(1)).
$$
Finally, we can prove a lower bound for the graphs $K_{r, t^{r-1}(r-1)!+1}$.

\begin{theorem}\label{Krt^(r-1)(r-1)!+1}
For arbitrary prime power $q$ and $t\mid q-1$ the graph $G(q^{r-1},t(q^{r-2}+\cdots + q + 1))$ is $K_{r,t^{r-1}(r-1)!+1}$-free.
\end{theorem}
\begin{proof}
Let $H$ and $H^\prime$ be subgroups of $\F^\ast$ orders $t(q^{r-1}+ \cdots+q+1)$ and $q^{r-1}+ \cdots+q+1$ respectively.
Choose $t$ coset representatives $g_1,\ldots, g_t$ of $H/H^\prime$.

 We follow the same steps as in Theorem~\ref{Kr(r-1)!+1}. The only difference is that in the final system we obtain equations of the form
$$
N\left( \frac{\alpha_j}{\beta_j}+\frac{h_2}{h_1} \right) =
 N\left(\frac{h_j}{h_1}\right)N(\beta_j)^{-1},
$$ 
but $N\left(\frac{h_j}{h_1}\right)$ for $j=2, \ldots,r$ can be any of the $t$ elements $N(g_1), \ldots, N(g_t)$ only. Therefore, in this case there are at most $t^{r-1}(r-1)!$ solutions for $h_2/h_1$. Again, each of these solutions uniquely defines a vertex $\la x,y\ra$ attached to all $\la a_i,b_i\ra$.
\end{proof}
We obtain an asymptotic bound
$$
\ex(n, K_{r,t^{r-1}(r-1)!+1})\geq \frac{1}{2}t^{\frac{r-1}{r}} n^{2-\frac{1}{r}}(1+o(1)).
$$

Therefore, our construction achieves the same asymptotic lower bounds for bipartite Tur\'an numbers as the Norm-graphs do.

\end{document}